\documentclass[12pt,a4paper,twoside]{article}
\usepackage[english]{babel}
\usepackage[utf8]{inputenc}
\usepackage[T1]{fontenc}
\usepackage{graphicx}
\usepackage{enumerate}
\usepackage{amsmath, amssymb, amsthm, dsfont,aligned-overset}
\usepackage[DIV=14,BCOR=1.5cm]{typearea}
\usepackage[makeroom]{cancel}
\usepackage{xcolor} 
\usepackage{cleveref}
\usepackage{apptools}
\usepackage{appendix}
\renewcommand{\d}{\mathrm{d}} 

\renewcommand{\H}{\mathcal{H}} 

\newcommand{\N}{\mathbb{N}} 
\newcommand{\Z}{\mathbb{Z}} 
\newcommand{\R}{\mathbb{R}} 
\newcommand{\process}[1]{\left(#1\right)_{t\ge 0}}

\DeclareMathOperator{\G}{\Gamma }
\DeclareMathOperator{\Pro}{\mathbb{P}}
\DeclareMathOperator{\E}{\mathbb{E}}
\DeclareMathOperator{\V}{\mathbb{V}}

\DeclareMathOperator{\Lamp}{\mathcal{L}}

\theoremstyle{definition}

\theoremstyle{plain}
\newtheorem{Definition}{Definition}[section]

\newtheorem{Theorem}[Definition]{Theorem}
\newtheorem{Lemma}[Definition]{Lemma}

\newtheorem{Corollary}[Definition]{Corollary}

\theoremstyle{remark}

\begin{document}

\title{Persistence probabilities of a smooth self-similar anomalous diffusion process}
\author{Frank Aurzada and  Pascal Mittenb\"uhler}
\date{\today}
\maketitle
\begin{abstract}
We consider the persistence probability of a certain fractional Gaussian process $M^H$ that appears in the Mandelbrot-van Ness representation of fractional Brownian motion. This process is self-similar and smooth. We show that the persistence exponent of $M^H$ exists and is continuous in the Hurst parameter $H$. Further, the asymptotic behaviour of the persistence exponent for $H\downarrow0$ and $H\uparrow1$, respectively, is studied. Finally, for $H\to 1/2$, the suitably renormalized process converges to a non-trivial limit with non-vanishing persistence exponent, contrary to the fact that $M^{1/2}$ vanishes.
\end{abstract}



\section{Introduction and main results}

This paper is concerned with the persistence exponent of a certain class of anomalous diffusion processes. Anomalous diffusion processes are an important tool in modelling physical systems \cite{bouchaud,metzler,rosso}. The persistence probability of a real-valued process $(X_t)_{t\ge 0}$ is given by
$$
\Pro\left[\sup_{t\in[0,T]} X_t<1\right].
$$
For self-similar processes, one expects the behaviour of this quantity to be of order $T^{-\theta(X)+o(1)}$, when $T\to\infty$, for some $\theta(X)\in(0,\infty)$. If this is the case we say that the persistence exponent of $X$ exists and equals $\theta(X)$.
We refer to \cite{bray2013persistence} for an overview on the relevance of this question to physical systems and to \cite{andersen2015persistence} for a survey of the mathematics literature.

In this work, we deal with anomalous diffusion processes, also called fractional processes, and this means we have to start by recalling what is presumably the most important fractional process, namely fractional Brownian motion (FBM): This is a continuous, centred Gaussian process $(B^H_t)_{t\ge 0}$ with covariance
$$
\E[ B^H_t B^H_s ] = \frac{1}{2} ( t^{2H} + s^{2H} - |t-s|^{2H}), \qquad t,s\ge 0,
$$
where $H\in(0,1)$ is the so-called Hurst parameter. For $H=\frac{1}{2}$, FBM is just usual Brownian motion, while for $H\neq\frac{1}{2}$ the process has stationary but non-independent increments.

The process of interest in this paper stems from the Mandelbrot-van Ness integral representation of fractional Brownian motion given by 
\begin{equation} \label{eq:VanNess2}
\sigma_H B_t^H = \int_0^t (t-s)^{H-\tfrac{1}{2}}\d B_s + \int_{-\infty}^0 (t-s)^{H-\tfrac{1}{2}}-(-s)^{H-\tfrac{1}{2}}\d B_s,
\end{equation}
where $(B_s)_{s\in\R}$ in the stochastic integral is a usual (two-sided) Brownian motion. The derivation of the normalisation constant
\begin{equation*}
\sigma_H\coloneqq\frac{\G(H+\tfrac{1}{2})}{\sqrt{2H\sin(\pi H)\G(2H)}}
\end{equation*}
can be found e.g.\ in Theorem 1.3.1 of \cite{mishura2008stochastic}.
The two processes appearing in the Mandelbrot-van Ness representation
$$
R_t^H:=\int_0^t (t-s)^{H-\tfrac{1}{2}}\d B_s \qquad \text{and}\qquad  M_t^H := \int_{-\infty}^0 (t-s)^{H-\tfrac{1}{2}}-(-s)^{H-\tfrac{1}{2}}\d B_s
$$
are independent. We stress that $R^H$ can be defined for all parameters $H>0$, while $M^H$ only makes sense for $H\in(0,1)$. We also note that for $H=\frac{1}{2}$, $R^{1/2}=B^{1/2}$ is a usual Brownian motion, while $M^{1/2}$ vanishes.

Further, let us mention that $B^H$, $R^H$, and $M^H$ are $H$-self-similar, respectively. It is simple to show that $B^H$ and $R^H$ have continuous versions, in fact even $\gamma$-H\"older continuous for any $\gamma<H<1$, while these processes are not $H$-H\"older continuous. Contrary, $M^H$ turns out to be a smooth (i.e.\ infinitely differentiable) process. Therefore, $M^H$ is a self-similar, but smooth process, which makes it an interesting object in modelling physical systems.

The persistence exponent of $B^H$ was obtained by Molchan \cite{molchan1999maximum} (for subsequent refinements see \cite{aecp,agpp,pengrao}) to the end that
$$
\theta(B^H)=1-H.
$$
The persistence exponent of fractionally integrated Brownian motion $R^H$ (also called Riemann-Liouville process) was the subject of the recent study \cite{aurzada2022asymptotics}. There, it was shown that the function $H\mapsto\theta(R^H)$ is continuous and tends to infinity when $H\downarrow0$. Further, the limiting behaviour when $H\to\infty$ is investigated in the papers \cite{ad,ps}.

We will thus turn our attention to the less studied process $M^H$ for $H\in(0,\tfrac{1}{2})\cup (\tfrac{1}{2},1)$ and ask for the existence of the persistence exponent
\begin{equation}\label{eq:thetaM}
    \theta(M^H)\coloneqq\lim_{T\rightarrow\infty}-\frac{1}{\log T}\log\Pro\left[\sup_{t\in[0,T]}M_t^H<1\right],
\end{equation}
its continuity properties as well as its asymptotic behaviour for $H\downarrow0$, $H\uparrow1$, and $H\to\frac{1}{2}$, respectively.
Apart from trying to understand the persistence behaviour of the fractional process $M^H$, the goal is to shed light on the relation of the persistence exponents of $B^H$ (studied in \cite{molchan1999maximum} and subsequent papers), $R^H$ (studied in \cite{ad,ps,aurzada2022asymptotics}), and $M^H$. 
The following two theorems on existence, continuity, and asymptotic behaviour of $\theta(M^H)$ are the main objective of this work.

\begin{Theorem}\label{conv}\label{Theorem:main}
The limit in \eqref{eq:thetaM} exists for any $H\in(0,\tfrac{1}{2})\cup (\tfrac{1}{2},1)$. It has the following asymptotic behaviour:
\begin{enumerate}
    \item[\textnormal{a)}] $\lim_{H\downarrow0}\frac{\theta(M^H)}{H}=1$
    \item[\textnormal{b)}] $\lim_{H\uparrow1}\frac{\theta(M^H)}{1-H}=1.$
\end{enumerate}
\end{Theorem}

As a side remark, we note that the persistence exponent of the process $M^H$ exhibits the same limiting behaviour at $0$ and $1$ as that of the integrated fractional Brownian motion, cf.\ Theorem~1 in \cite{aurzada2022asymptotics}. We have no explanation for this coincidence at this point.

The next theorem deals with the situation at $H=\frac{1}{2}$. It shows that the persistence exponent, as a function of $H$, can be continuously extended to $(0,1)$, i.e.\ including the point $H=\frac{1}{2}$. At $H=\frac{1}{2}$, the value of the continuous extension turns out to be positive, which is surprising given that $M^{1/2}$ vanishes. There is a non-trivial limit process $M^{\ast,1/2}$, whose persistence exponent corresponds to the value of the continuous extension of $H\mapsto\theta(M^H)$ at $H=\frac{1}{2}$.

\begin{Theorem}
\label{thm:continuity}
The mapping $H\mapsto \theta(M^H)$ is continuous on $(0,\frac{1}{2})\cup (\frac{1}{2},1)$ and continuously extendable to the whole interval $(0,1)$ with strictly positive limit at $H=\tfrac{1}{2}$. The persistence exponent of the following process is the value of the continuous extension of $H\mapsto\theta(M^H)$ at $H=\frac{1}{2}$:
\begin{equation} \label{eqn:defnmast}
    M^{\ast,1/2}_t:=\int_0^\infty \log\left(1+\frac{t}{s}\right) \d B_s.
\end{equation}
\end{Theorem}

The proof of these results is similar in methodology to \cite{aurzada2022asymptotics}. The first step is to transfer to the problem to the stationary setup via time-changing the process: Define the stationary Gaussian process (GSP):
$$
(\Lamp M^H)_\tau:=\frac{1}{\sqrt{\V M^H_1}} \, e^{-H\tau} M^H_{e^\tau},\qquad \tau\in\R.
$$
It is called the Lamperti transform of $M^H$. Note that one has to exclude the trivial case $H=\frac{1}{2}$ here, as then $\V M^{\frac{1}{2}}_1 = 0$. We note that since $M^H$ is a centred, continuous, $H$-self-similar Gaussian process, its Lamperti transform $\Lamp M^H$ is a centred, continuous GSP of unit variance.

The first goal is to prove that
\begin{equation} \label{eqn:pessandstat}
\theta(M^H)=\lim_{T\to\infty}- \frac{1}{T} \log \Pro\left[\sup_{\tau\in[0,T]} (\Lamp M^H)_\tau <0\right],
\end{equation}
where the right hand side is also called the persistence exponent of the GSP $\Lamp M^H$. This will be achieved in Lemma~\ref{lem:pecoincide} below using Theorem 1 in \cite{molchan2006problem}. We can then work in the setup of GSPs and focus on the correlation function of $\Lamp M^H$.

In order to prove the subsequent main results we will rely on a continuity lemma for the persistence exponent of GSPs developed in \cite{dembomukherjee1,dembomukherjee2,aurzada2023persistence} and summarised in Lemma 1 in \cite{aurzada2022asymptotics}. This continuity lemma relates the \emph{convergence of the correlation functions} of a sequence of centred, continuous GSPs to the \emph{convergence of their persistence exponents}, subject to checking some technical conditions. The continuity of the function $H\mapsto\theta(M^H)$ on $(0,\frac{1}{2},\frac{1}{2},1)$ follows directly from the continuity lemma after checking its conditions. The asymptotic behaviour for $H\downarrow0$, $H\uparrow1$, and $H\to 1/2$, respectively, is obtained by rescaling the correlation function of $\Lamp M^H$ appropriately.

Let us outline the structure of this paper. In Section~\ref{chapter:preliminaries}, we are going to set up some preliminary material and prove the existence of the limit in (\ref{eq:thetaM}) and the relation (\ref{eqn:pessandstat}). The proof for the continuity in $(0,\frac{1}{2})\cup(\frac{1}{2},1)$ is given in Section~\ref{chapter:H0}. Afterwards, the proofs for the asymptotic behaviour for $H\downarrow0$ and $H\uparrow1$ are given in Section~\ref{chapter:H1}. Finally, the situation for $H\to \frac{1}{2}$ is the subject of Section~\ref{chapter:H1/2}.

\section{Preliminaries}\label{chapter:preliminaries}

\subsection{The continuity lemma}
At the heart of our analysis lies Lemma 1(a) from \cite{aurzada2022asymptotics} (developed in \cite{dembomukherjee1,dembomukherjee2,aurzada2023persistence}), which allows for a connection between \emph{convergence of correlation functions} of GSPs and \emph{convergence of persistence exponents}. For the reader's convenience the mentioned lemma is restated here. 
\begin{Lemma}\label{KilianLemma}
For $k\in\N$, let $(Z_\tau^{k})_{\tau\geq0}$ be a centred GSP with correlation function $A_k:\R^+_0\rightarrow[0,1]$ and $A_k(0)=1$. Suppose that the sequence of functions $(A_k)_{k\in\N}$ converges pointwise for $k\rightarrow\infty$ to a correlation function $A:\R^+_0\rightarrow[0,1]$ corresponding to a GSP $(Z_\tau)_{\tau\geq0}$.
If $Z^{k}$ and $Z$ have continuous sample paths and the conditions 
\begin{align}
    &\lim_{L\rightarrow\infty}\limsup_{k\rightarrow\infty}\sum_{\tau=L}^\infty A_k\left(\frac{\tau}{\ell}\right)=0, \quad\text{for every } \ell\in\N,\label{sum}\\
    &\limsup_{\epsilon\downarrow0}\lvert\log(\epsilon)\rvert^\eta \sup_{k\in\N,\tau\in[0,\epsilon]}(1-A_k(\tau))<\infty\quad\text{for some }\eta>1,\label{log}\\
    &\limsup_{\tau\rightarrow\infty}\frac{\log A(\tau)}{\log \tau}<-1\label{frac}
\end{align}
are fulfilled, then 
\begin{equation*}
    \lim_{k,T\rightarrow\infty}\frac{1}{T}\log \Pro\left[Z_\tau^{k}<0, \forall \tau\in[0,T]\right]= \lim_{T\rightarrow\infty}\frac{1}{T}\log \Pro\left[Z_\tau<0, \forall \tau\in[0,T]\right].
\end{equation*}
\end{Lemma}

\subsection{The correlation function of $\Lamp M^H$}
The goal of this subsection is to give some convenient representations of the correlation function of the GSP $\Lamp M^H$. For $t\in\R$ and $H\in(0,\tfrac{1}{2})\cup(\tfrac{1}{2},1)$ define the functions
\begin{equation*}
    k_t^H(s):= (t+s)^{H-\tfrac{1}{2}}-s^{H-\tfrac{1}{2}}, \qquad s\in\R,
\end{equation*}
and note that a distributionally equivalent version of $\process{M_t^H}$ is given by
$$
    M_t^H=\int_0^\infty k_t^H(s)\d B_s,\qquad t\geq 0.
$$
We note for future reference that for any $t\ge 0$ and  $H\in(0,\tfrac{1}{2})$ the function $k^H_t$ is non-positive, while it is non-negative for $H\in(\tfrac{1}{2},1)$.

Further, we not only look at the Lamperti transform of $M^H$, but also consider the Lamperti transforms of $B^H$ and $R^H$:
\begin{align*}
    (\Lamp B^H)_\tau&:= e^{-H\tau}B^H_{e^\tau},\qquad \tau\in\R,\\
    (\Lamp R^H)_\tau&:= \sqrt{2H}e^{-H\tau}R^H_{e^\tau},\qquad \tau\in\R,\\
    (\Lamp M^H)_\tau&= \left(\sigma_H^2-\tfrac{1}{2H}\right)^{-\tfrac{1}{2}}e^{-H\tau}M^H_{e^\tau},\qquad \tau\in\R,
\end{align*}
where the normalisation is such that $\V[ (\Lamp B^H)_\tau]=\V[ (\Lamp R^H)_\tau]=\V[ (\Lamp M^H)_\tau]=1$ for all $\tau\in\R$ (and we used \eqref{eq:VanNess2} and the independence of $R^H$ and $M^H$ to obtain the correct normalisation for $M^H$ by calculating that $0<\V[M^H_1]=\sigma^2_H-\frac{1}{2H}$).
The corresponding correlation functions are given by 
\begin{align*}
    c_H(\tau)&\coloneqq \E[(\Lamp B^H)_\tau(\Lamp B^H)_0] = \cosh(H\tau)-\tfrac{1}{2}\left(2\sinh(\tfrac{\tau}{2})\right)^{2H},\\
    r_H(\tau)&\coloneqq \E[(\Lamp R^H)_\tau (\Lamp R^H)_0]= \frac{4H}{1+2H}\, e^{-\tfrac{\tau}{2}}\,{}_{2}F_1\left(1,\tfrac{1}{2}-H,\tfrac{3}{2}+H,e^{-\tau}\right),
\end{align*}
with the standard notation for the Gaussian hypergeometric function ${}_2F_1$
(and we used the integral representation of ${}_2F_1$ and the fact that ${}_2F_1(a,b,c,z)={}_2F_1(b,a,c,z)$).
The correlation function of $\Lamp M^H$ can be derived using equation \eqref{eq:VanNess2}, the independence of $R^H$ and $M^H$, and the fact that $B^H$, $R^H$ and $M^H$ are centred processes. This gives the following representations.

\begin{Lemma} \label{lem:representationsofg_H}
We have
\begin{equation}\label{ugg}
    g_H(\tau)\coloneqq\E[(\Lamp M^H)_\tau (\Lamp M^H)_0]= (\sigma_H^2-\tfrac{1}{2H})^{-1}(\sigma_H^2c_H(\tau)-\tfrac{1}{2H}r_H(\tau)).
\end{equation}
Alternatively, we have
\begin{equation}\label{integralrepresentation}
    g_H(\tau)=(\sigma_H^2-\tfrac{1}{2H})^{-1}\int_0^\infty K_0^H(s)K_\tau^H(s)\d s
\end{equation}
as well as the relation
\begin{equation}\label{const=int}
    \sigma_H^2-\tfrac{1}{2H}=\int_0^\infty K_0^H(s)^2 \d s,
\end{equation}
where
\begin{equation*}
    K_\tau^H(s):=e^{-H\tau}\left( (e^\tau+s)^{H-\tfrac{1}{2}}-s^{H-\tfrac{1}{2}}\right),\qquad s\geq 0.
\end{equation*}
\end{Lemma}

Again we note for future reference that, similarly to the function $k_t^H$, for any $\tau\geq 0$ and  $H\in(0,\tfrac{1}{2})$ the function $K^H_\tau$ is non-positive, while it is non-negative for $H\in(\tfrac{1}{2},1)$.
Then, by positivity of the integrands in both cases \eqref{integralrepresentation} and \eqref{const=int} we also obtain positivity of $g_H$. 

\subsection{Connecting the persistence exponents}
The purpose of this subsection is to show that the respective persistence exponents of the process $M^H$ and its Lamperti transform $\Lamp M^H$ exist and are identical for any $H\in(0,\tfrac{1}{2})\cup(\tfrac{1}{2},1)$, i.e.\ (\ref{eqn:pessandstat}) holds, so that we can focus our attention on the exponents of the latter process.\\
We need the following corollary which is a consequence of Theorem~1 in \cite{molchan2006problem}. 
\begin{Corollary}\label{Cor:Mol}
Let $\process{X_t}$ be a centred, continuous, $H$-self-similar Gaussian process with positive covariance function satisfying, for some $c>0$
\begin{align} \label{eqn:masterthesis1.3}
    \E\left[ |X_t-X_{t'}|^2\right] \leq c |t-t'|^{2H},\qquad t,t'\in[0,1].
\end{align}
Let $(\H_X,\lVert\cdot\rVert_X)$ be the associated Reproducing kernel Hilbert space (RKHS). If there exists $\phi\in\H_X$ such that for all $t\geq1$ also $\phi(t)\geq1$ holds, then the persistence exponents of $X$ and the Lamperti transform of $X$ both exist and coincide, i.e.\ 
$$
    \theta(X):=\lim_{T\rightarrow\infty}\frac{1}{\log T} \log\Pro\left[  \sup_{t\in[0,T]} X_t<1 \right]
    = \lim_{T\rightarrow\infty}\frac{1}{T}\log\Pro\left[  \sup_{\tau\in[0,T]} e^{- \tau H} X_{e^\tau}<0 \right].
$$
\end{Corollary}
\begin{proof}
We use the following special case of Theorem~1 in \cite{molchan2006problem}: $U_0=[0,1]$, $S_0=\{0\}$, $\Delta=[0,1]$, and Molchan's $\phi_T$ is our $T$-independent function $\phi$. Further $\psi(T)=\log T$, while $\sigma_T$ is a sufficiently large constant.

Let us verify the conditions (a), (b), (c) in \cite{molchan2006problem}: (a) is precisely our assumption that $\phi(t)\geq 1$, for all $t\geq 1$, and the fact that the RKHS norm of $\phi$ is constant in $T$ and thus in $o(\psi(T))$. Condition (b) is straightforward to check. Only condition (c) is non-trivial. Here, the first step is to note that $\sup\{ (\E \left[X_s^2\right])^{1/2} : s \in[0,1]\}$ is a constant. Further, the function $\delta_T(h)=\delta(h)=\sup\{ (\E |X_t-X_{t'}|^2)^{1/2} : t,t' \in[0,1], |t-t'|\leq h\}$ satisfies $\delta(h)\leq c h^H$, by assumption (\ref{eqn:masterthesis1.3}). This shows that $|\int_0^1 \delta(h) \d \sqrt{\log 1/h}|<\infty$, yielding (c) for a sufficiently large constant $\sigma_T$. 
The theorem then implies (using continuity of paths in the second step): 
\begin{align*}
    \lim_{T\rightarrow\infty}\frac{1}{\log T} \,\log \Pro\left[  \sup_{t\in[0,T]} X_t<1 \right]     &=\lim_{T\rightarrow\infty}\frac{1}{\log T}\,\log \Pro\left[\forall t\in (1,T] :\,X_t\neq0 ,\, X\rvert_{\lbrace1,T\rbrace}<0\right]\\
    &=\lim_{T\rightarrow\infty}\frac{1}{\log T}\log \Pro\left[  \sup_{t\in[1,T]} X_t<0 \right]\\
    &= \lim_{T\rightarrow\infty}\frac{1}{T}\,\log\Pro\left[  \sup_{\tau\in[0,T]} e^{-H\tau} X_{e^\tau}<0 \right].\qedhere
\end{align*}
\end{proof}

In order to apply the last lemma to $M^H$, we have to check that (\ref{eqn:masterthesis1.3}) is satisfied.

\begin{Lemma} \label{lem:addlemma1.3}
The process $M^H$ satisfies (\ref{eqn:masterthesis1.3}).
\end{Lemma}

\begin{proof} Let $t'<t$ and observe that
\begin{align*}
\E\left[|M^H_t-M_{t'}^H|^2\right] & = \int_0^\infty \left( (t+s)^{H-\tfrac{1}{2}} - (t'+s)^{H-\tfrac{1}{2}}\right)^2 \d s
\notag \\
& = \int_{t'}^\infty \left( (t-t'+s)^{H-\tfrac{1}{2}} - s^{H-\tfrac{1}{2}}\right)^2 \d s
\notag \\
& \leq (t-t')^{2H-1} \int_0^\infty \left( \left(1+\frac{s}{t-t'}\right)^{H-\tfrac{1}{2}} - \left(\frac{s}{t-t'}\right)^{H-\tfrac{1}{2}}\right)^2 \d s
\notag \\
& = (t-t')^{2H} \int_0^\infty \left( \left(1+s\right)^{H-\tfrac{1}{2}} - s^{H-\tfrac{1}{2}}\right)^2 \d s
\notag \\
& = c_H (t-t')^{2H}. \label{eqn:1.3computation} \qedhere
\end{align*}
\end{proof}

We can now apply Corollary~\ref{Cor:Mol} to our process $M^H$.

\begin{Lemma} \label{lem:pecoincide} Fix $H\in(0,\tfrac{1}{2})\cup(\tfrac{1}{2},1)$.
The persistence exponent of $M^H$ exists and satisfies
$$
\theta(M^H):=\lim_{T\to\infty}- \frac{1}{\log T} \log \Pro\left[\sup_{t\in[0,T]} M^H_t <1\right]=\lim_{T\to\infty}- \frac{1}{T} \log \Pro\left[\sup_{\tau\in[0,T]} (\Lamp M^H)_\tau <0\right],
$$
i.e.\ (\ref{eqn:pessandstat}) holds.
\end{Lemma}
\begin{proof}
In this proof the conditions of Corollary \ref{Cor:Mol} will be verified in order to show the claim for $X=M^H$.
Clearly, the process is continuous and $H$-self-similar and satisfies (\ref{eqn:masterthesis1.3}), by Lemma~\ref{lem:addlemma1.3}. 
It is only left to show that for any $H\in(0,\tfrac{1}{2})\cup (\tfrac{1}{2},1)$ there exists a function $\phi\in\H_{M^H}$ with $\phi(t)\geq1$ for any $t\geq1$.
A function $\phi$ in the RKHS can be parametrized by an auxiliary function $f_H\in L^2(\R^+,\d u)$ such that
\begin{equation*}
    \phi(t)=\int_0^\infty k_t^H(u)f_H(u)\d u.
\end{equation*}
    In the case $H\in(0,\tfrac{1}{2})$, a suitable auxiliary function $f_H$ is given by
\begin{equation*}
    f_H(u):=(\sigma_H^2-\tfrac{1}{2H})^{-1}k_1^H(u),
\end{equation*}
which is square integrable since the process $M^H$ is of finite variance.
For $t\geq1$ we can conclude
\begin{align}
    \phi(t)&=(\sigma_H-\tfrac{1}{2H})^{-1} \int_0^\infty \left(s^{H-\tfrac{1}{2}}- (t+s)^{H-\tfrac{1}{2}}\right)\left( s^{H-\tfrac{1}{2}}- (1+s)^{H-\tfrac{1}{2}} \right)\d s \notag \\
    &\geq(\sigma_H-\tfrac{1}{2H})^{-1}\int_0^\infty\left( s^{H-\tfrac{1}{2}}- (1+s)^{H-\tfrac{1}{2}} \right)^2\d s \notag \\
    &=g_H(0)
    =1. \label{eqn:copyexistenceofphi}
\end{align}
Turning to $H\in(\tfrac{1}{2},1)$, we need to change the auxiliary function $f_H$ to
\begin{equation*}
    f_H(u):=\begin{cases}(\sigma_H^2-\tfrac{1}{2H})^{-1}(2H-1)u^{H-\tfrac{3}{2}},\quad&\text{for }u>\frac{\sigma_H^2-\tfrac{1}{2H}}{2} \\0,\quad\quad&\text{otherwise}.
    \end{cases}
\end{equation*}
This is again a valid auxiliary function since again $f_H\in L^2(\R^+,\d u)$ holds.
Then we can estimate for all $s>0$ and $t\geq1$ 
\begin{equation*}
    0\leq  \frac{(t+s)^{H-\tfrac{1}{2}}-s^{H-\tfrac{1}{2}}}{t} = \frac{H-\frac{1}{2}}{t} \int_0^t (u+s)^{H-\frac{3}{2}}\d u  \leq  (H-\tfrac{1}{2})s^{H-\tfrac{3}{2}}.
\end{equation*}
This implies for $C_H\coloneq\sigma_H^2-\tfrac{1}{2H}$ the chain of inequalities
\begin{align*}
    \phi(t)&=C_H^{-1} \int_{\tfrac{C_H}{2}}^\infty \left( (t+s)^{H-\tfrac{1}{2}}-s^{H-\tfrac{1}{2}} \right)(2H-1)s^{H-\tfrac{3}{2}}\d s\\
    &\geq2 C_H^{-1} t^{-1} \int_{\tfrac{C_H}{2}}^\infty\left( (t+s)^{H-\tfrac{1}{2}}-s^{H-\tfrac{1}{2}} \right)^2\d s\\
    &=2 C_H^{-1} t^{2H-2} \int_{\tfrac{C_H}{2}}^\infty\left( \left(1+\frac{s}{t}\right)^{H-\tfrac{1}{2}}-\left(\frac{s}{t}\right)^{H-\tfrac{1}{2}} \right)^2\d s\\
    &=2 C_H^{-1} t^{2H-1} \int_{\tfrac{C_H}{2t}}^\infty\left( \left(1+s\right)^{H-\tfrac{1}{2}}-s^{H-\tfrac{1}{2}} \right)^2\d s\\
    &\geq2 C_H^{-1} t^{2H-1} \left(\int_0^\infty\left( \left(1+s\right)^{H-\tfrac{1}{2}}-s^{H-\tfrac{1}{2}} \right)^2\d s-\frac{C_H}{2t}\right)\\
    &=2 t^{2H-1}- t^{2H-2}\\
    &\overset{t\geq1}{\geq} 1,
\end{align*}
where we used in the second to last estimate that $(1+s)^{H-\tfrac{1}{2}} -s^{H-\tfrac{1}{2}} \leq 1$. The proof is completed by applying Corollary \ref{Cor:Mol}.
\end{proof}

\section{Continuity of $H\mapsto\theta(M^H)$} \label{chapter:H0}
\subsection{Estimates for $H\neq\frac{1}{2}$}
In this section, we summarise some estimates on the correlation function $g_H$ that will be used in the following sections. For improved readability we introduce the function
\begin{equation}\label{eq:tildedef}
    \Tilde{\sigma}^2(H):= 2H\sigma_H^2 =\frac{\Gamma(H+\tfrac{1}{2})^2}{\sin(\pi H)\Gamma(2H)},\qquad H\in(0,\tfrac{1}{2})\cup(\tfrac{1}{2},1),
\end{equation}
which turns equation \eqref{ugg} into
\begin{equation}\label{eq:tilderep}
     g_H(\tau)= (\Tilde{\sigma}^2(H)-1)^{-1}(\Tilde{\sigma}^2(H)c_H(\tau)-r_H(\tau)).
\end{equation}
Note that for any $H\in(0,\tfrac{1}{2})\cup(\tfrac{1}{2},1)$ this can be simplified, first by applying Euler's reflection $\Gamma(z)\Gamma(1-z)=\frac{\pi}{\sin(\pi z)}$, $z\not\in\Z$, and then the Legendre duplication formula $\Gamma(z)\Gamma(z+\tfrac{1}{2})=2^{1-2z}\sqrt{\pi}\Gamma(2z)$, $z>0$, to see that
\begin{equation}\label{eq:logconv}
        \Tilde{\sigma}^2(H)=\pi^{-\tfrac{1}{2}}\,2^{1-2H}\Gamma(H+\tfrac{1}{2})\Gamma(1-H).
\end{equation}
The next lemma will be used to show continuity of $\theta(M^H)$ for all $H\in(0,\tfrac{1}{2})\cup(\tfrac{1}{2},1)$.
\begin{Lemma}\label{Lemma:sigmatilde}
The function $\Tilde{\sigma}^2$ as defined in \eqref{eq:tildedef} is strictly convex, attains 
its minimum in $H=\tfrac{1}{2}$ for the value $\Tilde{\sigma}^2(\tfrac{1}{2})=1$ and exhibits the asymptotic behaviour
\begin{equation*}
    \lim_{H\uparrow1}\Tilde{\sigma}^2(H)=\infty, \quad \lim_{H\downarrow0}\Tilde{\sigma}^2(H)=2.
\end{equation*}
More precisely, $\Tilde{\sigma}^2(H)\sim (4(1-H))^{-1}$ for $H\uparrow 1$. 
\end{Lemma}
\begin{proof}
We get $\Tilde{\sigma}^2(\tfrac{1}{2})=1$ by a simple evaluation of the function using $\Gamma(\tfrac{1}{2})=\sqrt{\pi}$. From the representation of $\Tilde{\sigma}^2(H)$ in equation \eqref{eq:logconv}, we get $\lim_{H\downarrow0}\Tilde{\sigma}^2(H)=2$. Similarly, from \eqref{eq:logconv} we obtain that for $H\uparrow 1$
$$
\Tilde{\sigma}^2(H) \sim \pi^{-\frac{1}{2}} 2^{-1} \Gamma\left(\frac{3}{2}\right)\,\frac{\Gamma(2-H)}{1-H}\sim\pi^{-\frac{1}{2}} 2^{-2} \Gamma\left(\frac{1}{2}\right)\,\frac{\Gamma(1)}{1-H} = \frac{1}{4(1-H)}.
$$
Let us finally show strict convexity. In order to achive this we show that the derivative vanishes only at $H=\frac{1}{2}$. 
Taking the logarithm of the expression \eqref{eq:logconv}, we get 
\begin{equation*}
    \log \Tilde{\sigma}^2(H) =-\tfrac{1}{2}\log(\pi)+\log(2)(1-2H)+\log(\Gamma(H+\tfrac{1}{2}))+\log(\Gamma(1-H)),
\end{equation*}
which is strictly convex by the Gamma function being strictly logarithmic convex. 
Investigating the logarithmic derivative yields
\begin{equation*}
    \partial_H\log \Tilde{\sigma}^2(H)=\frac{\Gamma'(H+\tfrac{1}{2})}{\Gamma(H+\tfrac{1}{2})}-\frac{\Gamma'(1-H)}{\Gamma(1-H)}-2\log 2.
\end{equation*}
We evaluate this for $H=\frac{1}{2}$ using the table in chapter $44:7$ of the book \cite{oldham2009atlas} which lists the values of the so-called Digamma function $\Psi$ defined by $\Psi(z)\coloneqq \frac{\Gamma'(z)}{\Gamma(z)}$.
    With Euler's constant $\gamma$, one finds the following values: $\Psi(1)=-\gamma$ and $\Psi(\frac{1}{2})=-\gamma -2\log 2$. Thus the logarithmic derivative vanishes at $H=\frac{1}{2}$ and since $\Tilde{\sigma}^2(\tfrac{1}{2})>0$ holds, we can deduce 
\begin{equation*}
    0=\partial_H\log \Tilde{\sigma}^2(H)\big\rvert_{H=\tfrac{1}{2}}=\frac{\partial_H\Tilde{\sigma}^2(H)}{\Tilde{\sigma}^2(H )}\Bigg\rvert_{H=\tfrac{1}{2}},
\end{equation*}
implying
    $\partial_H\Tilde{\sigma}^2(H)\big\rvert_{H=\tfrac{1}{2}}=0$. 
\end{proof}
We also need an estimate for $c_H$, which is provided in the next lemma.
\begin{Lemma}\label{Lemma:c_H2}
For $H\in(0,\tfrac{1}{2})\cup (\tfrac{1}{2},1)$ and $\tau\geq0$ the following inequality holds:
\begin{equation*}
    c_H(\tau)\leq  \tfrac{1}{2}e^{-\tau H} +e^{-\tau (1-H)}.
\end{equation*}
\end{Lemma}
\begin{proof}
We first see from the definition of $c_H$ that
\begin{align*}
    2 c_H(\tau)&=e^{-\tau H}+e^{\tau H}-\left(e^{\tfrac{\tau}{2}}-e^{-\tfrac{\tau}{2}}\right)^{2H}=e^{-\tau H}+e^{\tau H}    \left(1-(1-e^{-\tau})^{2H}\right).
\end{align*}
For $H>\frac{1}{2}$, we use Bernoulli's inequality 
$(1-e^{-\tau})^{2H}\geq1-2He^{-\tau}$, while for $H<\frac{1}{2}$ and any $x\in[0,1]$  we have $x^{2H}\geq x$ so that
\begin{align*}
    1-(1-e^{-\tau})^{2H}\leq\begin{cases}1-(1-e^{-\tau})&\leq2e^{-\tau}\quad\textnormal{ for }H\in(0,\frac{1}{2}),
    \\
    1-(1-2He^{-\tau})&\leq2e^{-\tau}\quad\textnormal{ for }H\in(\frac{1}{2},1).
    \end{cases}
\end{align*}
Then we get by reassembling
\begin{equation*}
    2 c_H(\tau)=e^{-\tau H}+e^{\tau H}    \left(1-(1-e^{-\tau})^{2H}\right)\leq e^{-\tau H}+2e^{-\tau(1- H)}.\qedhere
\end{equation*}
\end{proof}
Combining Lemmas~\ref{Lemma:sigmatilde} and~\ref{Lemma:c_H2} we obtain the following lemma, that will be used to show the technical condition \eqref{sum}.

\begin{Lemma}\label{Lemma:limsupsum}
Fix $H_0\in[0,\tfrac{1}{2})\cup(\tfrac{1}{2},1]$. There exist $\Delta_{H_0}\in(0,1)$ and $\delta_{H_0}>0$ such that
\begin{enumerate}
    \item[a)]
    for any $\tau\geq 0$ and $H\in(H_0-\delta_{H_0},H_0+\delta_{H_0})\cap(0,\tfrac{1}{2})\cup (\tfrac{1}{2},1)$:
    \begin{equation*}
    g_H(\tau)\leq  \tfrac{4}{\Delta_{H_0}}e^{-\tau H(1-H)};
\end{equation*}
    \item[b)]
    for any function $\kappa:(0,\tfrac{1}{2})\cup(\tfrac{1}{2},1)\rightarrow\R^+$, $\tau\geq 0$, and $L\in\N$:
    \begin{equation*}
    \limsup_{H\rightarrow H_0}\sum_{\tau=L}^\infty g_H\left(\frac{\tau}{\kappa(H)}\right)\leq \limsup_{H\rightarrow H_0}\, \frac{4\kappa(H)}{\Delta_{H_0}H(1-H)} e^{-\tfrac{(L-1)H(1-H)}{\kappa(H)}}.
\end{equation*}
\end{enumerate}
\end{Lemma}
\begin{proof}
By Lemma \ref{Lemma:sigmatilde} we can choose for each $H_0\in[0,\tfrac{1}{2})\cup (\tfrac{1}{2},1]$ a $\delta_{H_0}>0$ such that there exists $0<\Delta_{H_0}<1$ with $\Tilde{\sigma}^2(H)\geq \Delta_{H_0}+1$ for any
$
H\in(H_0-\delta_{H_0},H_0+\delta_{H_0})\cap ((0,\tfrac{1}{2})\cup (\tfrac{1}{2},1))
$.
Using this together with Lemma~\ref{Lemma:c_H2} and representation \eqref{eq:tilderep}, we get
\begin{equation*}
    g_H(\tau)\leq \frac{\Tilde{\sigma}^2(H)}{\Tilde{\sigma}^2(H)-1}\, c_H(\tau)\leq  \tfrac{2}{\Delta_{H_0}}\left(e^{-\tau H}+e^{-\tau(1-H)} \right)\leq  \tfrac{4}{\Delta_{H_0}}e^{-\tau H(1-H)}.
\end{equation*}
From this we get
\begin{align*}
    \limsup_{H\rightarrow H_0}\sum_{\tau=L}^\infty g_H  \left(\tfrac{\tau}{\kappa(H)}\right)&\leq\limsup_{H\rightarrow H_0}\sum_{\tau=L}^\infty \tfrac{4}{\Delta_{H_0}} e^{-\tfrac{\tau H(1-H)}{\kappa(H)}}\\
    &\leq\limsup_{H\rightarrow H_0}  \tfrac{4}{\Delta_{H_0}}\int_{L-1}^\infty e^{-\tfrac{\tau H(1-H)}{\kappa(H)}}\d \tau\\
    &=\limsup_{H\rightarrow H_0} \tfrac{4\kappa(H)}{\Delta_{H_0}H(1-H)}e^{-\tfrac{(L-1)H(1-H)}{\kappa(H)}}.\qedhere
\end{align*}
\end{proof}
The next lemma is used to show the technical condition \eqref{log}.
\begin{Lemma}\label{Lemma:int}
For $H\in(0,\tfrac{1}{2})\cup (\tfrac{1}{2},1)$ and $\tau\geq0$ the following inequality holds:
\begin{equation}
    g_H(\tau)\geq e^{-\tau H}.
\end{equation}
\end{Lemma}
\begin{proof}
We first notice that for $H\in(0,1)$ and any $u\geq0$ the function 
\begin{equation*}
    \R^+_0\rightarrow\R^+_0,\, x\mapsto \left\lvert u^{H-\tfrac{1}{2}}-(x+u)^{H-\tfrac{1}{2}}\right\rvert = \left\lvert H-\frac{1}{2}\right\rvert\cdot \int_0^x (z+u)^{H-\frac{3}{2}}\d z
\end{equation*}
is increasing and since the product $K_\tau^H K_0^H$ is always positive we can estimate
\begin{equation*}
     K_0^H(u)K_\tau^H(u)=\left\lvert K_0^H(u)\right\rvert \, e^{-\tau H}\, \left\lvert u^{H-\tfrac{1}{2}}-(e^\tau+u)^{H-\tfrac{1}{2}}\right\rvert\geq e^{-\tau H} K_0^H(u) ^2.
\end{equation*}
This implies for any $\tau\geq0$
\begin{align*}
    g_H(\tau)&=\left(\int_0^\infty K_0^H(u)^2\d u\right)^{-1}\int_0^\infty K_0^H(u)K_\tau^H(u)\d u \geq e^{-\tau H}. \qedhere
\end{align*}
\end{proof}
\subsection{Continuity of $\theta(M^H)$ for $H\neq\tfrac{1}{2}$}

\begin{proof}[Proof of Theorem~\ref{thm:continuity}, part 1 of 3] We prove continuity of the function $H\mapsto\theta(M^H)$ on $(0,\tfrac{1}{2})\cup (\tfrac{1}{2},1)$.

The goal is to apply Lemma~\ref{KilianLemma} to the sequence of correlation functions $A_H(\tau):=g_H(\tau)$ for $H\to H_0$, where $A_\infty(\tau):=g_{H_0}(\tau)$. Since the correlation functions $g_H(\tau)$ are continuous in $H$ for each point $\tau$, we only have to verify the technical conditions of Lemma~\ref{KilianLemma}.

For any $H_0\in(0,\tfrac{1}{2})\cup (\tfrac{1}{2},1)$ by Lemma \ref{Lemma:limsupsum} there exist $\Delta_{H_0}>0$ and $0<\delta_{H_0}<\min(H_0,1-H_0)$ such that
    for any $\ell,L\in\N$ we get
\begin{equation*}
    \sup_{H\in(H_0-\delta_{H_0},H_0+\delta_{H_0})}\sum_{\tau=L}^\infty g_H(\tfrac{\tau}{\ell})\leq\sup_{H\in(H_0-\delta_{H_0},H_0+\delta_{H_0})} \tfrac{4\ell}{\Delta_{H_0}H(1-H)}e^{-\tfrac{(L-1)H(1-H)}{\ell}},
\end{equation*}
which converges to zero for $L\to\infty$, showing \eqref{sum}.
    Further,  by Lemma \ref{Lemma:int},
\begin{align}
    \log(\epsilon)^2 \sup_{H\in(H_0-\delta_{H_0},H_0+\delta_{H_0}),\tau\in[0,\epsilon]}(1-g_H(\tau))
    &\leq \log(\epsilon)^2\left(1-e^{-\epsilon(\delta_{H_0}+H_0)}\right) \notag \\
    &\leq (\delta_{H_0}+H_0)\log(\epsilon)^2 \epsilon,
    \label{eqn:analogously}
\end{align}
which converges to $0$ for $\epsilon\to 0$ thus showing condition \eqref{log} for $\eta=2$.
   To verify condition \eqref{frac} we use Lemma \ref{Lemma:limsupsum} $a)$ to see that for $\tau>1$
\begin{equation*}
    \frac{\log g_{H_0}(\tau)}{\log \tau}\leq \frac{\log\left(\frac{4}{\Delta_{H_0}}\right)}{\log \tau} -\frac{\tau H_0(1-H_0)}{\log \tau},
\end{equation*}
which converges to $-\infty$ for $\tau\to\infty$. Thus, the claim follows from Lemma \ref{KilianLemma}.
\end{proof}
\section{Asymptotics of $\theta(M^H)$}\label{chapter:H1}
\subsection{Asymptotics for $H\downarrow0$}
The goal of this section is to prove Theorem~\ref{Theorem:main} a). We start with a technical lemma. 
\begin{Lemma}\label{Lemma:2F1}
For $H\in(0,\tfrac{1}{2})$ and $\tau>0$ the following inequality holds:
\begin{equation*}
    1\leq{}_{2}F_1\left(1,\tfrac{1}{2}-H,\tfrac{3}{2}+H,e^{-\tau}\right)\leq\frac{\Gamma(H+\tfrac{3}{2})}{\Gamma(\tfrac{3}{2}-H)\Gamma(2H+1)}\,(1-e^{-\tau})^{-1}.
\end{equation*}
\end{Lemma}
\begin{proof}
The first inequality follows from the series representation of the hypergeometric function, as all terms in the series are non-negative (because $H<\frac{1}{2}$).
For the second inequality we use the integral representation of the hypergeometric function (see e.g. equation 60:3:3 in \cite{oldham2009atlas})
and estimate
\begin{align*}
    {}_{2}F_1\left(1,\tfrac{1}{2}-H,\tfrac{3}{2}+H,e^{-\tau}\right)&=\frac{\Gamma(H+\tfrac{3}{2})}{\Gamma(\tfrac{1}{2}-H)\Gamma(2H+1)}\int_0^1t^{-H-\tfrac{1}{2}}(1-t)^{2H}(1-e^{-\tau}t)^{-1}\d t\\
    &\leq\frac{\Gamma(H+\tfrac{3}{2})}{\Gamma(\tfrac{1}{2}-H)\Gamma(2H+1)}\int_0^1t^{-H-\tfrac{1}{2}} (1-e^{-\tau})^{-1}\d t\\
    &=\frac{\Gamma(H+\tfrac{3}{2})}{\Gamma(\tfrac{3}{2}-H)\Gamma(2H+1)}\,(1-e^{-\tau})^{-1}.\qedhere
\end{align*}
\end{proof}
We have collected all the necessary material to give the proof of Theorem \ref{Theorem:main} a).
\begin{proof}[Proof of Theorem \ref{Theorem:main} a)]
Our goal is to apply Lemma~\ref{KilianLemma}. Here we look at the sequence of correlation functions $A_H(\tau):=g_H(\tfrac{\tau}{H})$ for $H\downarrow 0$. We are going to show that $A_H(\tau)\to A_\infty(\tau):=e^{-\tau}$ pointwise and that the technical conditions of Lemma~\ref{KilianLemma} are satisfied. This yields that the persistence exponents of the GSPs corresponding to $A_H$ converge to the persistence exponent of the Ornstein-Uhlenbeck process, which equals $1$ (as can be obtained by direct computation, cf.\ \cite{slepian1962one}, or by using the fact that the Ornstein-Uhlenbeck process is the Lamperti transform of Brownian motion).
Since $g_H$ is the correlation function of $\Lamp M^H$, $A_H$ is the correlation function of $((\Lamp M^H)_{\tau/H})$ so that the persistence exponent corresponding to $A_H$ equals $\theta(M^H)/H$, as the following computation shows:
\begin{equation} \label{eqn:changeofpe}
\lim_{T\to\infty} \frac{1}{T} \log \Pro\left[ \sup_{\tau\in[0,T]} (\Lamp M^H)_{\tau/H} \right] = \lim_{T\to\infty} \frac{1/H}{T/H} \log \Pro\left[ \sup_{\tau\in[0,T/H]} (\Lamp M^H)_{\tau} \right] = \theta(M^H)/H.
\end{equation}
Let us therefore finish the proof with the verification of the application of Lemma~\ref{KilianLemma}:
{\it Step 1:} Pointwise convergence. By Lemma \ref{Lemma:2F1},  for $H\downarrow0$,
\begin{align*}
    1\leq{}_{2}F_1\left(1,\tfrac{1}{2}-H,\tfrac{3}{2}+H,e^{-\tfrac{\tau}{H}}\right)\leq\frac{\Gamma(H+\tfrac{3}{2})}{\Gamma(\tfrac{3}{2}-H)\Gamma(2H+1)}\,(1-e^{-\tfrac{\tau}{H}})^{-1}\rightarrow 1,
\end{align*}
from which we deduce that
\begin{equation*}
        r_H(\tfrac{\tau}{H})=\tfrac{4H}{1+2H}\, e^{-\tfrac{\tau}{2H}}\,{}_{2}F_1\left(1,\tfrac{1}{2}-H,\tfrac{3}{2}+H,e^{-\tfrac{\tau}{H}}\right) \to 0.
\end{equation*}
Further, it is immediate that for $H\downarrow0$ one has $2c_H(\tfrac{\tau}{H}) \to  e^{-\tau}$, which in combination with the result $\Tilde{\sigma}^2(H)\to 2$ for $H\downarrow 0$ in Lemma~\ref{Lemma:sigmatilde} yields
$$
    g_H(\tfrac{\tau}{H})= (\Tilde{\sigma}^2(H)-1)^{-1}(\Tilde{\sigma}^2(H)c_H(\tfrac{\tau}{H})-r_H(\tfrac{\tau}{H})) \to  e^{-\tau}.
$$    
{\it Step 2:} Verification of the technical conditions of Lemma~\ref{KilianLemma}. First, condition \eqref{frac} is easily verified with $A_\infty(\tau)=e^{-\tau}$.
By Lemma \ref{Lemma:limsupsum} b) we get for the choice $\kappa(H):=\ell H$ for any $\ell\in\N$
\begin{equation*}
    \limsup_{H\downarrow 0}\sum_{\tau=L}^\infty g_H(\tfrac{\tau}{\ell H})\leq\limsup_{H\downarrow 0}\frac{4\ell}{\Delta_{0}(1-H)} e^{-\tfrac{(L-1) (1-H)}{\ell}}=\frac{4\ell}{\Delta_{0}} e^{-\tfrac{(L-1)}{\ell}},
\end{equation*}
which converges to $0$ for $L\to\infty$, showing \eqref{sum}.
Lastly, analagously to (\ref{eqn:analogously}) above, we can show \eqref{log}  using Lemma~\ref{Lemma:int}. 
\end{proof}
 
\subsection{Asymptotics for $H\uparrow1$}
Similarly to the last section, the goal of this section is to prove Theorem~\ref{Theorem:main} b). Again, we start with a technical lemma. 
\begin{Lemma}\label{Lemma:Beta}
There exists a $\delta>0$ such that for any $H\in(1-\delta,1)$ we have
\begin{equation}\label{eq:Betaest}
    \frac{1}{4}\frac{H-\tfrac{1}{2}}{1-H}\geq \sigma_H^2-\frac{1}{2H}.
\end{equation}
\end{Lemma}
\begin{proof}
Using (\ref{eq:logconv}), we can see that (\ref{eq:Betaest}) is equivalent to
$$
(H-1)^2-\frac{1}{2} (H-1) +\frac {1}{2}- \frac{2^{2(1-H)}}{\sqrt{\pi}} \Gamma(2-H) \Gamma\left(H+\frac{1}{2}\right) \geq 0.
$$
We claim that even
$$
~~~~~~~~~~~~ -\frac{1}{2} (H-1) +\frac {1}{2}- \frac{2^{2(1-H)}}{\sqrt{\pi}} \Gamma(2-H) \Gamma\left(H+\frac{1}{2}\right) \geq 0
$$
for $H$ close to $1$. We use the Taylor expansions:
\begin{eqnarray*}
2^{2(1-H)} &=& e^{(1-H) 2 \log(2)} = 1 + (1-H) 2 \log(2) + O((1-H)^2),
\\
\Gamma(2-H) &=& \Gamma(1) - \Gamma'(1)(H-1) + O((1-H)^2)=1+\Gamma'(1)(1-H) + O((1-H)^2),
\\
\Gamma\left(H+\frac{1}{2}\right) &=& \Gamma\left(\frac{3}{2}\right) + \Gamma'\left(\frac{3}{2}\right)(H-1) + O((1-H)^2) =  \frac{\sqrt{\pi}}{2}  - \Gamma'\left(\frac{3}{2}\right)(1-H) + O((1-H)^2).
\end{eqnarray*}
Inserting this gives
\begin{eqnarray}
&&-\frac{1}{2} (H-1) +\frac {1}{2}- \frac{2^{2(1-H)}}{\sqrt{\pi}} \Gamma(2-H) \Gamma\left(H+\frac{1}{2}\right) \notag
\\
&=&\frac{1}{2} (1-H) +\frac {1}{2}  \notag
\\
&& -\frac{1}{\sqrt{\pi}} \left[1 + (1-H) 2 \log(2) \right]\left[1+ \Gamma'(1)(1-H)\right]\left[  \frac{ \sqrt{\pi}}{2} - \Gamma'\left(\frac{3}{2}\right)(1-H) \right] + O((1-H)^2)  \notag
\\
&=&\frac{1}{2} (1-H) +\frac {1}{2}-\frac{1}{\sqrt{\pi}} \left( \frac{ \sqrt{\pi}}{2} - (1-H)\left(\Gamma'\left(\frac{3}{2}\right)-\frac{\sqrt\pi}{2} \Gamma'(1)-\frac{\sqrt\pi}{2}2\log 2 \right)\right)
+ O((1-H)^2)  \notag
\\
&=&(1-H)\left( \frac{1}{2}  +\frac{1}{\sqrt\pi} \Gamma'\left(\frac{3}{2}\right)-  \frac{1}{2} \Gamma'(1)-\log 2  \right) \notag
+ O((1-H)^2)
\\
&=&(1-H)\left( \frac{3}{2}  -2\log 2\right)
+ O((1-H)^2). \label{eqn:newrs55}
\end{eqnarray}
Here we used the tables of chapters $43$:$7$ and $44$:$7$ of \cite{oldham2009atlas} to calculate
$$
    \Gamma'\left(\frac{3}{2}\right)=\Psi\left(\frac{3}{2}\right)\Gamma\left(\frac{3}{2}\right)=\frac{\sqrt{\pi}}{2}(2-\gamma-2\log 2 ),\qquad \Gamma'(1)=\Psi(1)\Gamma(1)=-\gamma.
$$
Since $\frac{3}{2}  -2\log 2>0$,
the term in (\ref{eqn:newrs55}) has to be positive for $H$ close to $1$.
\end{proof}
The following estimate gives a lower bound for $g_H(\tau)$, which is used to show convergence.
\begin{Lemma}\label{Conjecture:int2}
There exists a $\delta>0$ such that for any $H\in(1-\delta,1)$ and $\tau\geq0$
\begin{equation*}
    g_H(\tau)\geq e^{-\tau(1-H)}.
\end{equation*}
\end{Lemma}
\begin{proof}
Let $\delta>0$ be the same as in Lemma \ref{Lemma:Beta} and fix $H\in(1-\delta,1)$.

{\it Step 1:} We start by showing that for any $b\geq1$
\begin{align}\label{eq:mainestimate}
     \int_0^\infty K_0^H(u) \left((b+u)^{H-\tfrac{1}{2}}-(1+u)^{H-\tfrac{1}{2}}\right)\d u\geq \left(b^{2H-1}-1\right)\int_0^\infty K_0^H(u)^2\d u .
\end{align}
The left hand side of equation \eqref{eq:mainestimate} equals
$$
   (H-\tfrac{1}{2})\int_0^\infty K_0^H(u)\int_1^b (x+u)^{H-\tfrac{3}{2}}\d x \d u\\
    =(H-\tfrac{1}{2})\int_0^\infty \int_1^b K_0^H(u)(x+u)^{H-\tfrac{3}{2}}\d x \d u.
$$
We further see the inequality (using $H\geq \frac{1}{2}$):
\begin{align}
K_0^H(u)&=(1+u)^{H-\tfrac{1}{2}}-u^{H-\tfrac{1}{2}}=(H-\tfrac{1}{2}) \int_0^1 (z+u)^{H-\frac{3}{2}} \d z \notag \\
&\geq (H-\tfrac{1}{2}) (1+u)^{H-\tfrac{3}{2}} \geq (H-\tfrac{1}{2}) (x+u)^{H-\tfrac{3}{2}},
\end{align}
for any $x\geq1$.
By combining the two results, we obtain
\begin{align*}
    \int_0^\infty K_0^H(u) \left((b+u)^{H-\tfrac{1}{2}}-(1+u)^{H-\tfrac{1}{2}}\right)\d u\geq (H-\tfrac{1}{2})^2\int_0^\infty \int_1^b (x+u)^{2H-3}\d x \d u.
\end{align*}
The order of integration can then be exchanged by Tonelli's Theorem and
\begin{align*}
    &\int_1^b\int_0^\infty  (x+u)^{2H-3}\d u \d x
    = \frac{1}{2-2H}\int_1^b x^{2H-2}\d x
    = \frac{b^{2H-1}-1}{(2-2H)(2H-1)} .
\end{align*}
This yields the inequality
\begin{align*}
    \int_0^\infty K_0^H(u) \left((b+u)^{H-\tfrac{1}{2}}-(1+u)^{H-\tfrac{1}{2}}\right)\d u&\geq(H-\tfrac{1}{2})^2\frac{b^{2H-1}-1}{(2-2H)(2H-1)}
    \\
    &=(b^{2H-1}-1)\, \frac{1}{4}\, \frac{H-\tfrac{1}{2}}{1-H}
    \\
    &\geq (b^{2H-1}-1)(\sigma_H^2-\tfrac{1}{2H})\\
    &=\left(b^{2H-1}-1\right)\int_0^\infty K_0^H(u)^2\d u,
\end{align*}
where we applied the estimate of Lemma \ref{Lemma:Beta} in the third step. 
{\it Step 2:} We show that for any $b\geq 1$
\begin{equation}\label{eq:nearlythere}
    b^{1-2H}\int_0^\infty K_0^H(u)\left((b+u)^{H-\tfrac{1}{2}}-u^{H-\tfrac{1}{2}}\right)\d u \geq \int_0^\infty K_0^H(u)^2\d u. 
\end{equation}
Indeed, using \eqref{eq:mainestimate} we obtain
\begin{align*}
    &b^{1-2H}\int_0^\infty K_0^H(u)\left((b+u)^{H-\tfrac{1}{2}}-u^{H-\tfrac{1}{2}}\right)\d u\\
    &\quad= b^{1-2H} \int_0^\infty K_0^H(u)\left((b+u)^{H-\tfrac{1}{2}}-(1+u)^{H-\tfrac{1}{2}}+(1+u)^{H-\tfrac{1}{2}}-u^{H-\tfrac{1}{2}}\right)\d u\\
    &\quad =b^{1-2H}\left(\int_0^\infty K_0^H(u)\left((b+u)^{H-\tfrac{1}{2}}-(1+u)^{H-\tfrac{1}{2}}\right)\d u+\int_0^\infty K_0^H(u)^2\d u\right)\\
    &\quad \geq b^{1-2H}\left((b^{2H-1}-1)\int_0^\infty  K_0^H(u)^2\d u+\int_0^\infty K_0^H(u)^2\d u\right)\\
    &\quad =\int_0^\infty  K_0^H(u)^2\d u.
\end{align*}
{\it Step 3:} We plug in the choice $b=e^{\tau}$ into (\ref{eq:nearlythere}) and obtain:
\begin{align*}
    g_H\left(\tau\right)&=\left(\int_0^\infty K_0^H(u)^2\d u\right)^{-1}e^{-\tau H}\int_0^\infty K_0^H(u)\left((e^\tau+u)^{H-\tfrac{1}{2}}-u^{H-\tfrac{1}{2}}\right)\d u\\
    &= \left(\int_0^\infty K_0^H(u)^2\d u\right)^{-1}e^{-\tau(1-H)}e^{\tau(1-2H)}\int_0^\infty K_0^H(u)\left((e^\tau+u)^{H-\tfrac{1}{2}}-u^{H-\tfrac{1}{2}}\right)\d u\\
    &\geq e^{-\tau(1-H)}.\qedhere
\end{align*}
\end{proof}
We have collected all the necessary material to give the proof of Theorem \ref{Theorem:main} b).

\begin{proof}[Proof of Theorem \ref{Theorem:main} $b)$]
Our goal is to apply Lemma~\ref{KilianLemma}. This time we look at the sequence of correlation functions $A_H(\tau):=g_H\left(\frac{\tau}{1-H}\right)$ for $H\uparrow 1$. We are going to show that $A_H(\tau)\to A_\infty(\tau):=e^{-\tau}$ pointwise and that the technical conditions of Lemma~\ref{KilianLemma} are satisfied. This yields that the sequence of persistence exponents of the GSPs corresponding to $A_H$, which are given by $\theta(M^H)/(1-H)$ (the proof of which is analogous to (\ref{eqn:changeofpe})), converge to the persistence exponent of the Ornstein-Uhlenbeck process, which equals $1$.

{\it Step 1:} Pointwise convergence. We use Lemma~\ref{Conjecture:int2} (for $H$ close to $1$), the fact that $r_H(\tau)\geq 0$, and Lemma~\ref{Lemma:c_H2} to see that
\begin{equation*}
    e^{-\tau} \leq g_H\left(\frac{\tau}{1-H}\right)\leq \frac{\Tilde{\sigma}_H^2}{\Tilde{\sigma}_H^2-1}\, c_H\left(\frac{\tau}{1-H}\right) \leq \frac{\Tilde{\sigma}_H^2}{\Tilde{\sigma}_H^2-1}\,\left( \tfrac{1}{2}e^{-\tfrac{\tau}{1-H}}+e^{-\tau}\right).
\end{equation*}
Letting $H\uparrow 1$ and recalling  Lemma~\ref{Lemma:sigmatilde} to see that $\frac{\Tilde{\sigma}_H^2}{\Tilde{\sigma}_H^2-1}\to 1$, we obtain that indeed $g_H\left(\frac{\tau}{1-H}\right)\to e^{-\tau}$.

{\it Step 2:} Verification of the technical conditions of Lemma \ref{KilianLemma}. By Lemma \ref{Lemma:limsupsum} b) there exists a $\Delta_1>0$ such that for the choice of $\kappa(H)=\ell (1-H)$ for arbitrary $\ell\in\N$ we get
\begin{equation*}
    \limsup_{H\uparrow 1}\sum_{\tau=L}^\infty g_H(\tfrac{\tau}{\ell (1-H)})\leq\limsup_{H\uparrow 1}\frac{2\ell}{\Delta_{1} H} e^{-\tfrac{(L-1)H}{\ell}} = \frac{2\ell}{\Delta_{1}} e^{-\tfrac{L-1}{\ell}}, 
\end{equation*}
which converges to zero for $L\to\infty$, showing \eqref{sum}.
Using Lemma \ref{Conjecture:int2}, \eqref{log} is straightforward. 
Condition \eqref{frac} is easily verified as $A_\infty(\tau)=e^{-\tau}$. 
\end{proof}

\section{Proofs for the case $H\rightarrow1/2$} \label{chapter:H1/2}
\subsection{Pointwise limit of the correlation functions}
The goal of this subsection is to obtain the pointwise limit of the correlation function of $\Lamp M^H$, i.e.\ of $g_H$ defined in Lemma~\ref{lem:representationsofg_H}, when $H\to\frac{1}{2}$.

\begin{Lemma}\label{Lemma:1/2}
For any $\tau\geq0$ we have:
\begin{equation*}
    \lim_{H\rightarrow\tfrac{1}{2}}g_{H}(\tau)=\frac{3}{\pi^2}\, e^{-\tfrac{\tau}{2}}\int_0^\infty \log(1+\tfrac{1}{u})\log(1+\tfrac{e^\tau}{u})\d u =: g_{\ast,\frac{1}{2}}(\tau).
\end{equation*}
\end{Lemma}

We postpone the proof of this lemma and start with a technical result concerning properties of the functions $K_\tau^H(u)=e^{-\tau H}\left((e^\tau+u)^{H-\tfrac{1}{2}}-u^{H-\tfrac{1}{2}}\right)$. These functions appear in the representation of $g_H$, cf.\ (\ref{ugg}), and we shall employ l'H\^{o}spital's rule in the course of the proof of Lemma~\ref{Lemma:1/2} which will require some technical preparation.

As a simplification of the proof of the next lemma, we note that any $\tau\geq0$, $H\in(0,\tfrac{1}{2})\cup(\tfrac{1}{2},1)$ and $u\in\R^+$ we have
\begin{equation}\label{eq:reduct}
    K_\tau^H(u)=e^{-\frac{\tau}{2}} K_0^H\left(\frac{u}{e^\tau}\right).
\end{equation}

\begin{Lemma}\label{Lemma:L2}
Fix $\tau\geq0$ and $H\in(\tfrac{1}{4},\tfrac{1}{2})\cup(\tfrac{1}{2},\tfrac{3}{4})$. There exists $f_{\tau}\in L^1(\R^+,\d u)$ such that:
\begin{enumerate}
    \item[a)] For any $k\in\lbrace0,1,2\rbrace$ and $u>0$  the derivatives $\partial_H^kK_{\tau}^H(u)$ exist and for $k=1$ exhibit the limiting behaviour
\begin{equation*}
    \lim_{H\rightarrow\tfrac{1}{2}}\partial_H K_\tau^H(u)=e^{-\tfrac{\tau}{2}}\log\left(1+\frac{e^\tau}{u}\right).
\end{equation*}
\item[b)] A representative of $f_{\tau}$ can be chosen to fulfill the inequality \begin{equation*}
    \lvert\partial_H^k(K_{\tau}^H K_{0}^H)(u)\rvert\leq f_{\tau}(u) \textnormal{ for almost every }u\geq0.
\end{equation*}
\item[c)] For any $\ell\in\N$ there exists an $L\in\N$ with
\begin{equation}
    \sum_{\tau= L}^\infty\int_0^\infty f_{\tau/\ell}(u)\d u<\infty.
\end{equation}
\end{enumerate}
\end{Lemma}
\begin{proof}
{\bf Proof of a).} Given \eqref{eq:reduct} we can focus on $\tau=0$. Observe that
\begin{align}
    \partial_H K_0^H(u)&=  \log(1+u)(1+u)^{H-\tfrac{1}{2}}-\log(u)u^{H-\tfrac{1}{2}} \label{eqn:A}\\
    &=\log(1+u^{-1})(1+u)^{H-\tfrac{1}{2}}+\log(u)K_0^H(u) \notag \\
    \partial_H^2 K_0^H(u)&=\log(1+u)^2(1+u)^{H-\tfrac{1}{2}}-\log(u)^2u^{H-\tfrac{1}{2}} \notag \\
    &=\log(1+u^{-1})\log(u(1+u))(1+u)^{H-\tfrac{1}{2}}+\log(u)^2 K_0^H(u). \notag 
\end{align}
From (\ref{eqn:A}) and \eqref{eq:reduct}, part a) follows directly.

{\bf Proof of b).} 
We divide this into the cases $u\in(0,1]$ and $u\in(1,\infty)$.\\
\textit{The case $u\in(0,1]$:} We start by estimating $(1+u)^{H-\tfrac{1}{2}}\leq2$ and since $H\in(\tfrac{1}{4},\tfrac{3}{4})$,
    \begin{align}\label{eq:f}
        \lvert K_0^H(u)\rvert&\leq(1+u)^{H-\tfrac{1}{2}}+u^{H-\tfrac{1}{2}}
        \leq 3u^{-\tfrac{1}{4}}.
    \end{align}
    Using \eqref{eq:f} and the same arguments of its deduction again, it can be seen by applying the estimate   
    $\log\left(1+u^{-1}\right)\leq \log\left(\frac{2}{u}\right)\leq 1+\lvert\log(u)\rvert$
    that 
    \begin{align}
        \lvert\partial_HK_0^H(u)\rvert&\leq \log(1+u^{-1})(1+u)^{H-\tfrac{1}{2}}+\lvert\log(u) K_0^H(u)\rvert \notag \\
        &\leq 2(1+\lvert\log(u)\rvert)+3\lvert\log(u)\rvert u^{-\tfrac{1}{4}} \notag \\
        &\leq 5(1+\rvert\log(u)\lvert)^2u^{-\frac{1}{4}}. \label{eqn:collect2}
    \end{align}
    Similarly, we deduce 
    \begin{align}
        \lvert\partial_H^2K_0^H(u)\rvert&\leq \log(1+u^{-1})\lvert\log(u(1+u))\rvert(1+u)^{H-\tfrac{1}{2}}+\log(u)^2\lvert K_0^H(u)\rvert \notag \\
        &\leq  2(1+\lvert\log(u)\rvert)^2+3\log(u)^2u^{-\tfrac{1}{4}} \notag \\
        &\leq5(1+\lvert\log(u)\rvert)^2u^{-\tfrac{1}{4}}. \label{eqn:collect3}
    \end{align}
   \textit{The case $u\in(1,\infty)$:} Observe that (using $H\in(\tfrac{1}{4},\tfrac{3}{4})$) 
    \begin{align}\label{eq:f2}
        \lvert K_0^H(u)\rvert&=\left| \int_0^1 \left(H-\frac{1}{2}\right)\, (z+u)^{H-\frac{3}{2}}\d z\right| \leq \lvert H-\tfrac{1}{2}\, \rvert \cdot u^{H-\tfrac{3}{2}}
        \leq 2 u^{-\tfrac{3}{4}}.
    \end{align}
    For the first derivative we see by the inequalities
    \begin{align}
        (1+u)^{H-\tfrac{1}{2}}&\leq (1+u)^{\tfrac{1}{4}}
        \leq2u^{\tfrac{1}{4}}\label{eq:plus}
    \end{align}
    and $\log(1+u^{-1})\leq u^{-1}$ that with \eqref{eq:f2} we can make the estimation
    \begin{align}
        \lvert\partial_HK_0^H(u)\rvert&\leq \log(1+u^{-1})(1+u)^{H-\tfrac{1}{2}}+\lvert\log(u) K_0^H(u)\rvert \notag \\
        &\leq 2u^{\tfrac{1}{4}}\log(1+u^{-1})+2\lvert\log(u)\rvert u^{-\tfrac{3}{4}} \notag \\
        &\leq 2(1+\log(u))u^{-\tfrac{3}{4}}. \label{eqn:collect4}
    \end{align}
Finishing with the estimate on the second derivative, by \eqref{eq:f2}, \eqref{eq:plus}, and
\begin{equation*}
    \log(u(1+u))\leq\log(2u^2)\leq2(1+\log(u)),
\end{equation*}
we get (using again $\log(1+u^{-1})\leq u^{-1}$):
    \begin{align}
        \lvert\partial_H^2K_0^H(u)\rvert&\leq\log(1+u^{-1})\log(u(1+u))(1+u)^{H-\tfrac{1}{2}}+\lvert\log(u)^2 K_0^H(u)\rvert \notag \\
        &\leq 4 (1+\log(u))u^{-\tfrac{3}{4}}+2\log(u)^2u^{-\tfrac{3}{4}} \notag \\
        &\leq 6 (1+\log u)^2 u^{-\frac{3}{4}}. \label{eqn:collect5}
    \end{align}
Putting (\ref{eq:f}), (\ref{eqn:collect2}), (\ref{eqn:collect3}) for $u\in(0,1]$ and (\ref{eq:f2}), (\ref{eqn:collect4}), (\ref{eqn:collect5}) for $u\in(1,\infty)$ shows that
\begin{equation*}
    \sup_{k\in\lbrace0,1,2\rbrace} \lvert\partial_H^k K_{0}^H(u)\rvert \leq  2^3(1+\lvert\log(u)\rvert)^2\left(\mathds{1}_{(0,1]}(u)u^{-\frac{1}{4}}+\mathds{1}_{(1,\infty)}(u)u^{-\frac{3}{4}}\right)=:f(u).
\end{equation*}

By \eqref{eq:reduct}, we can extend this estimate to
\begin{align*}
    \lvert\partial_H^k(K_{\tau}^H K_{0}^H)(u)\rvert&\leq2^k f(u)e^{-\frac{\tau}{2}}f\left(\frac{u}{e^\tau}\right),
\end{align*}
which holds for $k\in\lbrace 0,1,2\rbrace$. By further estimating the expression on the right hand side, we arrive at the following estimate: 
\begin{align*}
    &2^k f(u)e^{-\frac{\tau}{2}}f\left(\frac{u}{e^\tau}\right)\\
    &\leq 2^8(1+\tau+\lvert\log(u)\rvert)^4\left(\mathds{1}_{(0,1]}(u)u^{-\frac{1}{2}}e^{-\frac{\tau}{4}}+\mathds{1}_{\left(1,e^\tau\right]}(u)u^{-1}e^{-\frac{\tau}{4}}+\mathds{1}_{\left(e^\tau,\infty\right)}(u)u^{-\frac{3}{2}}e^{\frac{\tau}{4}}\right)\\
    &\eqcolon f_\tau(u).
\end{align*}
This function is clearly $u$-integrable for any $\tau\geq0$.

{\bf Proof of c).} We first want to change from the summation depending on $\ell$ and $L$ to an integral estimate that is independent of both. To achieve this we use the fact that for any $\ell\in\N$ we can find an $L\in\N$ such that the function $(L-1,\infty)\rightarrow\R^+_0,\,\tau\mapsto f_{\frac{\tau}{\ell}}(u)$ is monotone for any $u\geq0$. Therefore by Tonelli's Theorem,
\begin{align*}
    \sum_{\tau=L}^\infty \int_0^\infty f_{\frac{\tau}{\ell}}(u)\d u
    \leq \int_0^\infty\int_{L-1}^\infty  f_{\frac{\tau}{\ell}}(u)\d\tau\d u\nonumber
    &\leq \ell\int_0^\infty\int_{0}^\infty  f_{\tau}(u)\d u\d \tau.
\end{align*}
Integrating the three different $u$-ranges in the definition of $f_\tau$, one ends up with the following expressions, respectively, which are each clearly $\tau$ integrable:
\begin{align*}
    &\sum_{k=0}^4\binom{4}{k}e^{-\frac{\tau}{4}}(1+\tau)^k\int_0^1\lvert\log(u)\rvert^{4-k}u^{-\frac{1}{2}}\d u,\\
    &\sum_{k=0}^4\binom{4}{k}e^{-\frac{\tau}{4}}(1+\tau)^k\int_1^{e^\tau}\lvert\log(u)\rvert^{4-k}u^{-1}\d u,\\
    &\sum_{k=0}^4\binom{4}{k}e^{\frac{\tau}{4}}(1+\tau)^k\int_{e^\tau}^\infty\lvert\log(u)\rvert^{4-k}u^{-\frac{3}{2}}\d u. \qedhere
\end{align*}
\end{proof}

Prepared with these technical facts, we can now determine the limit of the correlation functions $g_H$ when $H\to\frac{1}{2}$.

\begin{proof}[Proof of Lemma~\ref{Lemma:1/2}]
As the function $H\mapsto K_\tau^H(u)$ is continuous on $(\tfrac{1}{4},\tfrac{3}{4})$ for any $\tau\geq0$ and any $u>0$, we have
\begin{equation}\label{eq:vanish}
    \lim_{H\rightarrow\tfrac{1}{2}} K_\tau^H(u)=0.
\end{equation}
By Lemma~\ref{Lemma:L2}b and dominated convergence this implies that for any $\tau\ge 0$
\begin{equation*}
    \lim_{H\rightarrow\tfrac{1}{2}}\int_0^\infty K_0^H(u)K_\tau^H(u)\d u =0.
\end{equation*} 
This allows us to apply the l'H\^{o}spital rule on the representation \eqref{integralrepresentation} and (\ref{const=int}) as follows
\begin{align}
    \lim_{H\rightarrow\tfrac{1}{2}}g_H(\tau)&=\lim_{H\rightarrow\tfrac{1}{2}}\left(\int_0^\infty K^H_0(u)^2\,\d u\right)^{-1}\int_0^\infty K^H_0(u)K^H_{\tau}(u)\d u
    \notag \\
    &=\lim_{H\rightarrow\tfrac{1}{2}}\left(\partial_H\int_0^\infty K^H_0(u)^2\,\d u\right)^{-1}\partial_H\int_0^\infty K^H_0(u)K^H_{\tau}(u)\d u. \label{eqn:10-18-+}
\end{align}
Since $\partial_H( K^H_0(u)K^H_{\tau}(u))$ has an integrable majorant (cf.\ Lemma~\ref{Lemma:L2}b), we can exchange the order of differentiation, integration as well as the limit in $H$ by the dominated convergence theorem. By applying equation \eqref{eq:vanish} in the last step we see that
\begin{align*}
    \lim_{H\rightarrow\tfrac{1}{2}}\partial_H\int_0^\infty K^H_0(u)K^H_{\tau}(u)\d u&=\int_0^\infty\lim_{H\rightarrow\tfrac{1}{2}}\partial_H( K^H_0(u)K^H_{\tau}(u))\d u\\
    &=\int_0^\infty\lim_{H\rightarrow\tfrac{1}{2}}\partial_H K^H_0(u)K^H_{\tau}(u)+ K^H_0(u)\partial_H K^H_{\tau}(u)\d u\\
    &=0.
\end{align*}
We thus need to apply l'H\^{o}spital's rule again, which yields in continuation of (\ref{eqn:10-18-+}):
\begin{align*}
    \lim_{H\rightarrow\tfrac{1}{2}}g_H(\tau)&=\lim_{H\rightarrow\tfrac{1}{2}}\left(\partial_H^2\int_0^\infty K^H_0(u)^2\,\d u\right)^{-1}\partial_H^2\int_0^\infty K^H_0(u)K^H_{\tau}(u)\d u.
\end{align*}
Analogously to the arguments above, we obtain using part a) that
\begin{align*}
    \lim_{H\rightarrow\tfrac{1}{2}}\partial_H^2\int_0^\infty K^H_0(u)K^H_{\tau}(u)\d u&=\int_0^\infty\lim_{H\rightarrow\tfrac{1}{2}}\partial_H^2( K^H_0(u)K^H_{\tau}(u))\d u\\
    &=\int_0^\infty\lim_{H\rightarrow\tfrac{1}{2}}\sum_{k=0}^2\binom{2}{k}\partial_H^k K^H_0(u)\partial_H^{2-k} K^H_{\tau}(u)\d u\\
    &=2\int_0^\infty\lim_{H\rightarrow\tfrac{1}{2}}\partial_H K^H_0(u)\partial_H K^H_{\tau}(u)\d u\\
    &=2e^{-\tfrac{\tau}{2}}\int_0^\infty\log\left(1+\frac{1}{u}\right)\log\left(1+\frac{e^\tau}{u}\right)\d u.
\end{align*}
For the latter integral at $\tau=0$ it is known that
\begin{equation*}
    \int_0^\infty \log(1+\tfrac{1}{u})^2\d u= \frac{\pi^2}{3},
\end{equation*}
which gives the normalization constant as well as the fact that the integral is finite and thus finishes the proof of the lemma.
\end{proof}

\subsection{Extending the continuity of $H\mapsto\theta(M^H)$ to $H=\tfrac{1}{2}$}
Since we have seen that the function mapping $H\in(0,\tfrac{1}{2})\cup(\tfrac{1}{2},1)$ for any $\tau\geq0$ to $g_H(\tau)$ can be continuously extended to $H=\tfrac{1}{2}$, we can utilise this in combination with Lemma~\ref{KilianLemma} similarly to the previous sections to show existence of a continuous extension of $H\mapsto\theta(M^H)$ to $H=\tfrac{1}{2}$. 
In preparation of showing the technical conditions of Lemma~\ref{KilianLemma}, we state the following lemma.
\begin{Lemma}\label{Lemma:last}
There exists $C>0$ such that for any $\tau>0$
\begin{equation*}
     g_{\ast,\tfrac{1}{2}}(\tau)\leq Ce^{-\frac{\tau}{6}}.
\end{equation*}
\end{Lemma}
\begin{proof}
We first note that for any $\delta\in[0,1)$ and $n\geq 1$ the following integral is finite 
\begin{align*}
    \int_0^\infty \log(1+\tfrac{1}{u})^{\delta+n}\d u<\infty.
\end{align*}
Then we can apply Young's inequality $a\cdot b\leq p^{-1} a^p+q^{-1} b^q$ with $p=3/2$ and $q=3$ to see 
\begin{align*}
    \frac{\pi^2}{3}\,g_{*,\frac{1}{2}}(\tau)&=\int_0^\infty e^{-\frac{\tau}{9}}\log(1+\tfrac{1}{u}) \cdot e^{-\frac{7}{9}\frac{\tau}{2}}\log(1+\tfrac{e^\tau}{u})\d u\\
    &\leq \frac{2}{3}e^{-\frac{\tau}{6}}\int_0^\infty \log\left(1+\frac{1}{u}\right)^{\frac{3}{2}}\d u+\frac{1}{3}e^{-\frac{7}{6}\tau}\int_0^\infty \log\left(1+\frac{e^\tau}{u}\right)^{3}\d u\\
    &\leq e^{-\frac{\tau}{6}}\left(\int_0^\infty \log\left(1+\frac{1}{u}\right)^{\frac{3}{2}}\d u+\int_0^\infty \log\left(1+\frac{1}{v}\right)^{3}\d v\right). \qedhere
\end{align*}
\end{proof}

\begin{proof}[Proof of Theorem~\ref{thm:continuity}, part 2 of 3]
The goal is to use Lemma~\ref{KilianLemma}, where we consider $A_H(\tau):=g_H(\tau)$, let $H\to \frac{1}{2}$, and have $A_\infty(\tau)=g_{\ast,\frac{1}{2}}(\tau)$. The pointwise convergence follows from the definition of $g_{\ast,\frac{1}{2}}$ in Lemma~\ref{Lemma:1/2}. It thus remains to verify the technical conditions of Lemma~\ref{KilianLemma}.

We start with condition (\ref{sum}). Analogous to the proof of Lemma~\ref{Lemma:1/2} we use part c) of Lemma~\ref{Lemma:L2} to legitimise the multiple exchanges in the order of limits and integration in the next computation: In particular using l'H\^{o}spital's rule, we obtain
  \begin{align*}
      \lim_{H\rightarrow\tfrac{1}{2}}\sum_{\tau=L}^\infty g_{H}(\tfrac{\tau}{\ell})&=\lim_{H\rightarrow\tfrac{1}{2}}  \left(\int_0^\infty K^H_0(u)^2\,\d u \right)^{-1} \sum_{\tau=L}^\infty \int_0^\infty  K^H_0(u)K^H_{\frac{\tau}{\ell}}(u) \d u
      \\
      &=\lim_{H\rightarrow\tfrac{1}{2}} \left(\partial^2_H \int_0^\infty K^H_0(u)^2\,\d u\right)^{-1}\partial^2_H \sum_{\tau=L}^\infty \int_0^\infty  K^H_0(u)K^H_{\frac{\tau}{\ell}}(u) \d u
      \\
      &= \left(\int_0^\infty\lim_{H\rightarrow\tfrac{1}{2}}\partial^2_H  K^H_0(u)^2\,\d u\right)^{-1}\sum_{\tau=L}^\infty \int_0^\infty \lim_{H\rightarrow\tfrac{1}{2}}\partial^2_H  K^H_0(u)K^H_{\frac{\tau}{\ell}}(u) \d u
      \\      
      &=\sum_{\tau=L}^\infty g_{*,\frac{1}{2}}(\tfrac{\tau}{\ell}).
  \end{align*}
  
  We then go on to use Lemma~\ref{Lemma:last} to see that for any $\ell\in\N$
    \begin{equation*}
        0\leq\lim_{L\rightarrow\infty}\lim_{H\rightarrow\tfrac{1}{2}}\sum_{\tau=L}^\infty g_{H}(\tfrac{\tau}{\ell})=\lim_{L\rightarrow\infty}\sum_{\tau=L}^\infty g_{\ast,\tfrac{1}{2}}(\tfrac{\tau}{\ell})\leq \lim_{L\rightarrow\infty}\sum_{\tau=L}^\infty C e^{-\frac{\tau}{6}}=0,
    \end{equation*}
so that we have verified (\ref{sum}). 
  Condition \eqref{log} is easily verified, as Lemma \ref{Lemma:int} 
  implies that for any $H\in(0,1)$ and any $\tau\geq0$ also $g_{H}(\tau)\geq e^{-\tau}$, which gives immediately \eqref{log}.
Finally, Lemma \ref{Lemma:last} implies that the estimate $g_{\ast,\tfrac{1}{2}}(\tau)\leq C e^{-\frac{\tau}{6}}$ holds. This immediately gives condition \eqref{frac}. \qedhere
\end{proof}

\begin{Lemma}
The correlation function of the GSP
$$
(\Lamp M^{\ast,\frac{1}{2}})_\tau:= \frac{1}{\sqrt{\V M^{\ast,1/2}_1}}\, e^{-\tau/2} M^{\ast,\frac{1}{2}}_{e^\tau},
$$ with $M^{\ast,\frac{1}{2}}$ defined in (\ref{eqn:defnmast}), is $g_{\ast,\frac{1}{2}}$.
The persistence exponents of $\Lamp M^{\ast,\frac{1}{2}}$ and $M^{\ast,\frac{1}{2}}$ coincide. More precisely,
$$
\theta(M^{\ast,\frac{1}{2}}):=\lim_{T\to\infty} \frac{1}{\log T} \, \log \Pro\left[ \sup_{t\in[0,T]} M^{\ast,\frac{1}{2}}_t \leq 1\right] = \lim_{T\to\infty} \frac{1}{T} \, \log \Pro\left[ \sup_{t\in[0,T]} (\Lamp M^{\ast,\frac{1}{2}})_\tau \leq 0\right].
$$
\end{Lemma}

\begin{proof} This follows from Corollary~\ref{Cor:Mol} once we have checked its conditions. Firstly, we note that $M^{\ast,\frac{1}{2}}$ is a continuous, $\frac{1}{2}$-self-similar, Gaussian process. Secondly, $M^{\ast,\frac{1}{2}}$ satisfies (\ref{eqn:masterthesis1.3}), as can be seen by exactly the same computations that one finds in the proof of Lemma~\ref{lem:addlemma1.3}. Thirdly, one has to check that there is a function $\phi$ in the RKHS of $M^{\ast,\frac{1}{2}}$ with $\phi(t)\geq 1$ for all $t\geq 1$. Such a function is given by
$$
\phi(t) := \left( \int_0^\infty \log\left(1+\frac{1}{u}\right)^2\d u\right)^{-1}\, \int_0^\infty \log\left(1+\frac{t}{u}\right)\log\left(1+\frac{1}{u}\right)\d u, 
$$
and $\phi(t)\geq 1$ for all $t\geq 1$ can  be checked by the exact same steps as in  (\ref{eqn:copyexistenceofphi}).
\end{proof}
We can now prove that that the persistence exponent $\Lamp M^{\ast,\frac{1}{2}}$ (which is the same as the one of $M^{\ast,\frac{1}{2}}$ according to the last lemma) does not vanish. Therefore, the (continuous extension of the) function $H\mapsto \theta(M^H)$ does not vanish at $H=\frac{1}{2}$. This is somehow surprising as the initial process $M^H$ does vanish at $H=\frac{1}{2}$. 

\begin{proof}[Proof of Theorem~\ref{thm:continuity}, part 3 of 3]
We prove strict positivity of the persistence exponent $\theta(M^{\ast,\frac{1}{2}})=\lim_{H\rightarrow\tfrac{1}{2}}\theta(M^H)$, the latter equality holding according to the second part of the proof.
By Lemma \ref{Lemma:last} we know that $ \int_0^\infty g_{\ast,\tfrac{1}{2}}(\tau)\d   \tau<\infty$
and therefore, by Lemma $3.2$ in \cite{aurzada2023persistence}, the persistence exponent corresponding to the correlation function $g_{\ast,\frac{1}{2}}$ is strictly positive.
\end{proof}

\nocite{*}
\bibliographystyle{alpha}

\end{document}